\documentclass[reqno]{amsart}

\usepackage{eucal,amsrefs,graphicx,mathrsfs}
\usepackage[all, knot]{xy}

\usepackage{amssymb,amscd,url}

\allowdisplaybreaks

\newtheorem{theorem}{Theorem}

\newtheorem{proposition}[theorem]{Proposition}
\newtheorem{corollary}[theorem]{Corollary}

\theoremstyle{definition}

\newtheorem*{definition}{Definition}
\newtheorem{example}[theorem]{Example}
\newtheorem*{acknowledgement}{Acknowledgements}

\theoremstyle{remark}
\newtheorem{remark}[theorem]{Remark}

\newcommand{\Ncal}{{\mathcal N}}

\renewcommand{\AA}{\mathbb{A}}

\newcommand{\CC}{\mathbb{C}}

\newcommand{\GG}{\mathbb{G}}
\newcommand{\NN}{\mathbb{N}}
\newcommand{\PP}{\mathbb{P}}
\newcommand{\QQ}{\mathbb{Q}}
\newcommand{\RR}{\mathbb{R}}
\newcommand{\ZZ}{\mathbb{Z}}

\DeclareMathOperator{\Mat}{Mat}

\DeclareMathOperator{\preper}{PrePer}
\DeclareMathOperator{\diag}{diag}

\newcommand{\Qbar}{{\bar{\QQ}}}

\newcounter{CaseCount}
\Alph{CaseCount}

\begin{document}

\title
      {Canonical Height Functions For Monomial Maps}
\date{\today}

\author{Jan-Li Lin}
\address{Department of Mathematics, Indiana University, Bloomington \\ IN 47405 \\ USA}
\email{janlin@indiana.edu}

\author{Chi-Hao Wang}
\address{Department of Mathematics, National Central University, Jhongli, Taoyuan 32001, Taiwan}
\email{chihaowang84@gmail.com}

\thanks{}

\maketitle

\begin{abstract}
We show that the canonical height function defined by Silverman~\cite{S1} does not have the Northcott finiteness property
in general. We develop a new canonical height function for monomial maps. In certain cases, this new canonical height
function has nice properties.
\end{abstract}


\section{Introduction}

Height functions measure the arithmetic complexity of certain algebraically defined objects.
They play a important role in Diophantine geometry, Diophantine approximation and arithmetic dynamics.
The theory of canonical heights for a morphism $f:\mathbb{P}^N\to \mathbb{P}^N$ is quite well known.

For a dominant rational map, it is more difficult to define and study canonical height functions.
A case which people have studied and understood is the case of regular affine automorphisms, see \cite{K1,K2,Lee}.
Recently, Silverman~\cite{S1} developed the theory of canonical height functions for general dominant rational maps.
He also studied the behavior of canonical height functions in the case of monomial maps.

More precisely, for a dominant rational map $\varphi:\PP^N\dashrightarrow \PP^N$, the first dynamical degree of $\varphi$
is defined as
\[
\delta_\varphi = \lim_{n\to\infty} \deg(\varphi^n)^{\frac{1}{n}}.
\]
Now assume that $\varphi$ is defined over $\Qbar$, one needs a strong conjecture that the following infimum
\[
\ell_\varphi = \inf \{ \ell\ge 0\ |\ \sup_{n\ge 1}\frac{\deg(\varphi^n)}{n^\ell\delta_\varphi^n}<\infty \}
\]
exists (It is also conjectured that the infimum is an integer satisfying $0\le \ell_\varphi\le N$, see \cite[Conjecture 2]{S1}).
Also, define $\PP^N(\Qbar)_\varphi$ to be the set of points in $\PP^N(\Qbar)$ whose forward image is always well defined.
Assuming that the conjecture is true,
then, for $P\in \PP^N(\Qbar)_\varphi$, we define the canonical height of $P$ with respect to $\varphi$ by
\[
\hat{h}_\varphi^+(P) = \limsup_{n\to\infty}\frac{1}{n^{\ell_\varphi}\delta_\varphi^\ell} h(\varphi^n(P)).
\]
Notice that our notation for canonical height is slightly different from the notation in \cite{S1}, where the author uses
$\hat{h}_\varphi(P)$.

For general dominant rational map, the above conjecture is still open. On the other hand, for monomial maps,
the conjecture is true, and the degree sequence is well-understood.
Given a matrix $A$ with integer entries we associate a selfmap
$\varphi_A$ on the algebraic torus $\GG_m^N(\Qbar)$
by the formula
\[
\varphi_A(x_1,\ldots,x_N) := (x_1^{a_{11}}\cdots x_N^{a_{1N}},\ \cdots\ ,x_1^{a_{N1}}\cdots x_N^{a_{NN}}).
\]

The map $\varphi_A : \GG_m^N(\Qbar) \to \GG_m^N(\Qbar) $ is called the monomial map associated to $A$. 
The dynamical degree of $\varphi_A$ equals 
the spectral radius of $A$ (see \cite[Theorem 6.2]{HP}, or \cite[Theorem 6.2]{Lin1})
and the value $\ell_{\varphi_A}$ is a number $0\le \ell_\varphi\le N-1$
determined by the Jordan form of $A$ (see \cite[Theorem 6.2]{Lin1}). The canonical height function for monomial maps is
studied in \cite{S1}.

An important property we want for the height function is the Northcott style finiteness property.
It states that there should be only finitely many points for given bounded height and bounded degree.
More precisely, for monomial maps $\varphi_A$ on $\GG_m^N(\Qbar)$, given $B>0$ and $D>0$, we would want the set
\[
\bigl\{ P\in \GG_m^N(\Qbar)\ \bigl|\ \hat{h}^+_A(P) < B \text{ and } [\QQ(P):\QQ]<D \bigr\}
\]
to be finite.
However, in this paper, we prove that the canonical height function does not always satisfy the
Northcott finiteness property, see Example~\ref{ex:small_height_dim_2} and Proposition~\ref{prop:small_height}.

Next, we try to modify the definition of canonical height so that it will have the desired property.
The construction we use is inspired by the earlier work of Silverman on K3 surfaces~\cite{S2}, and
the later work of Kawaguchi and Lee on regular affine automorphisms~\cite{K1,K2,Lee}.
That is, we not only look at the height growth for the forward orbit, but we also look at the backward orbit.
The problem is that a monomial map $\varphi_A$ is generally not birational, hence we need to make some modification
to define the backward canonical height, then the total canonical height is defined to be the sum of the
forward and the backward canonical heights.

In section~\ref{sec:tot:can:ht}, we define and study the total canonical height function.
It has a uniqueness property (Theorem~\ref{thm:unique}), and for a class of monomial maps,
it is bounded below by the height of the point (Theorem~\ref{theorem:LB1}).
Thus, we have a Northcott finiteness property and a lower bound estimate for total canonical height
for this class of monomial maps.
This class of maps
includes all diagonalizable matrices in dimension 2, and a major class of matrices in dimensions
3 and 4.

For a non-diagonalizable matrix $A$, the canonical height function of $\varphi_A$ has some strange behavior.
We study a case of such maps thoroughly in section~\ref{sec:non_diag}. Namely, we study the case where $A$
has only one (real) eigenvalue. We show that, in this case, the map $\varphi_A$ will preserve a fibration, and
the canonical height function is constant on each fiber, i.e., it only depends on the base.
This also shows how the geometry of the map controls the arithmetic.

Finally, in the last section, we show that the total canonical height function we proposed is still not
the ultimate solution to all monomial maps. Therefore, a more refined theory of canonical height function
for monomial maps is still needed to be developed.

\begin{acknowledgement}
The paper was developed when both authors visited the Institute for Computational and Experimental Research in Mathematics (ICERM).
We would like to thank ICERM for the hospitality and support.
We would also like to thank Liang-Chung Hsia and Joseph H. Silverman for helpful discussions and comments.
\end{acknowledgement}


\section{Properties of Canonical Height Functions}

First, we define some notations. Throughout this paper,
we write $\Mat_{N}^{+}(\mathbb{Z})$ for the $N \times N$ matrices with integer coefficients
and nonzero determinant. We use the notation
\[
\diag(\lambda_1,\cdots,\lambda_N)
\]
to denote the diagonal matrix with diagonal entries $\lambda_1,\cdots,\lambda_N$.
We always use $P=(x_1,\cdots,x_N)$ to denote a point in $\GG_m^N(\Qbar)$.
%
For simplicity, we write $\delta_A=\delta_{\varphi_A}$
and $\ell_A = \ell_{\varphi_A}$ for a monomial map $\varphi_A$;
the number $\ell_A+1$ is the size of the largest Jordan blocks for those eigenvalues of $A$
with modulus $\delta_A$.

Recall from the introduction that,
for a point $P=(x_1,\ldots,x_N) \in \mathbb{G}_m^N(\Qbar)$,
the {\em canonical height of $P$ with respect to $\varphi_A$} is defined as
\[
\hat{h}_{\varphi_A}^{+}(P)=\hat{h}_{A}^{+}(P):=\limsup_{n\rightarrow \infty}
\frac{1}{n^{\ell_{A}}\delta_{A}^n} h(\varphi_A^n(P)).
\]

\begin{example}
\label{ex:easy}
A simple example is for $d>0$ and $A=d\cdot I_N$, we have
$\varphi_A(x_1,\cdots,x_N)=(x_1^d,\cdots,x_N^d)$, and since
\[
h(x_1^d,\cdots,x_N^d) = d\cdot h(x_1,\cdots,x_N),
\]
we have
\[
\hat{h}_{A}^{+}(P)=\limsup_{n\rightarrow \infty}
\frac{h(P^{d^n})}{d^n} =h(P).
\]
\end{example}

\begin{example}
\label{ex:easy_too}
We still assume $d>0$, but now let $A = -d\cdot I_N$. Then we have
$\varphi_A(x_1,\cdots,x_N)=(x_1^{-d},\cdots,x_N^{-d})$, thus
\begin{eqnarray*}
\varphi_A^n(x_1,\cdots,x_N) &=&
\left\{
  \begin{array}{ll}
    (x_1^{d^n},\cdots,x_N^{d^n}), &  \text{if $n$ is even;} \\
    (x_1^{-d^n},\cdots,x_N^{-d^n}), & \text{if $n$ is odd.}
  \end{array}
\right.
\end{eqnarray*}
On the other hand, the dynamical degree $\delta_A=|-d|=d$.
The sequence $\{ \frac{h(\varphi_A(P))}{d^n} \}_{n=1}^\infty$ has two limit point, namely,
$h(P)$ and $h(P^{-1})$. Therefore, the canonical height is the maximum of the two, i.e.,
\[
\hat{h}_{A}^{+}(P)= \max\{ h(P),h(P^{-1})\}.
\]
\end{example}

Several properties for the canonical height function has been proved by Silverman in \cite{S1}:
\begin{enumerate}
\item $0\le \hat{h}_{\varphi}^{+}(P) <\infty$.
\item $\hat{h}_{\varphi}^{+}(\varphi(P)) = \delta_\varphi\cdot\hat{h}_{\varphi}^{+}(P)$.
\item If $P\in\preper(\varphi)$, then $\hat{h}_{\varphi}^{+}(\varphi(P))=0$
\item Suppose $A\in\Mat_N^+(\ZZ)$ is a matrix with $\rho(A)>1$ and suppose the characteristic polynomial
of $A$ is irreducible in $\QQ[x]$, then for $P\in\GG_m^N(\Qbar)$,
\[
P\in\preper(\varphi_A)\Longrightarrow \hat{h}_{\varphi}^{+}(P)=0.
\]
\end{enumerate}

Nest, we turn to the Northcott finiteness property.
The following example shows that, even in $\GG_m^N(\QQ)$ (that is, degree $=1$), there
can be infinitely many points with arbitrarily small canonical height.

\begin{example}
\label{ex:small_height_dim_2}
Let
$A=
\left(
  \begin{array}{cc}
    2 & 1 \\
    1 & 1 \\
  \end{array}
\right)$,
so $\varphi_A(x,y)=(x^2 y,xy)$, and
the eigenvalues of $A$ are $\lambda_1=\frac{3+\sqrt{5}}{2}$ and $\lambda_2=\frac{3-\sqrt{5}}{2}$.
Notice that
\[
A^n =
\left(
  \begin{array}{cc}
    a_{1,1}(n) & a_{1,2}(n) \\
    a_{2,1}(n) & a_{2,2}(n) \\
  \end{array}
\right)
=
\left(
  \begin{array}{cc}
    \frac{(5+\sqrt{5})\lambda_1^n+(5-\sqrt{5})\lambda_2^n}{10}
    & \frac{\lambda_1^n-\lambda_2^n}{\sqrt{5}} \\
    \frac{\lambda_1^n-\lambda_2^n}{\sqrt{5}}
    & \frac{(5-\sqrt{5})\lambda_1^n+(5+\sqrt{5})\lambda_2^n}{10} \\
  \end{array}
\right)
\]
Therefore, we have the following
    \begin{eqnarray*}
    \hat{h}_{A}^{+}(P)
    & = & \limsup_{n \rightarrow \infty}\sum_{v \in M_K}\max\Bigl\{0,
          \frac{a_{1,1}(n)}{\lambda_1^n}\log\|x\|_v
          +\frac{a_{1,2}(n)}{\lambda_1^n}\log\|y\|_v,\\
    &&    \hspace{3.4cm}\frac{a_{2,1}(n)}{\lambda_1^n}\log\|x\|_v
          +\frac{a_{2,2}(n)}{\lambda_1^n}\log\|y\|_v \Bigr\}\\
    & = & \frac{1}{\sqrt{5}} \sum_{v \in M_K}\max\Bigl\{0,
          \frac{\sqrt{5}+1}{2}\log\|x\|_v +\log\|y\|_v\Bigr\}.
    \end{eqnarray*}

Since $\frac{\sqrt{5}+1}{2}$ is irrational, for all $\varepsilon > 0$,
there exists integers $y_1, y_2$ such that $|y_1\frac{\sqrt{5}+1}{2}+y_2|<\varepsilon$.
Let $P=(2^{y_1},2^{y_2})$, then
\[
\hat{h}_{A}^{+}(P) \leq  \frac{\log(2)}{\sqrt{5}}\varepsilon.
\]
Also, observe that $P$ is not preperiodic. Hence, by
\cite[Corollary 31]{S1}, $\hat{h}_{A}^{+}(P)>0$. In fact, if we
replace the $2$ in the definition of $P$ by another prime number, we can obtain infinitely many
such points with small (but nonzero) canonical height, using a similar argument.
\end{example}

Generalizing the above example, we obtain the following proposition in arbitrary dimension.

\begin{proposition}
\label{prop:small_height}
Given $A\in\Mat_N^+(\ZZ)$, suppose its characteristic polynomial is irreducible, and all the eigenvalues are
distinct and positive. Then, for any $\varepsilon >0$, there are infinitely many $P\in\GG_m^N(\QQ)$ with
$0 < \hat{h}_{A}^{+}(P) < \varepsilon$.
\end{proposition}

\begin{proof}
We write the matrix $A$ as $A=B\Lambda B^{-1}$, where
\[
\Lambda = \diag(\lambda_1,\cdots,\lambda_N)
\]
is diagonal with $\lambda_1> \cdots > \lambda_N>0$,
$B=(b_{i,j})$, and $B^{-1}=(c_{i,j})$.

Notice that the matrices $B$ and $B^{-1}$ are not unique.
However, denote $ A^n=\big(a_{i,j}(n)\big)$. By the relation
\[
a_{i,j}(n)=\sum_{k=1}^N b_{i,k}c_{k,j}\lambda_k^n,
\]
we can use Cramer's rule to give a formula for $b_{i,k}c_{k,j}$. Hence the numbers $b_{i,k}c_{k,j}$
does not depend on the choice of $B$, only depend on $A$.
Hence, we can define the constant
\[
R:=\max_{i,j,k}\{ |b_{i,k}c_{k,j}| \},
\]
which only depends on the matrix $A$.

Let $K$ be the splitting field of the characteristic polynomial of $A$.
Notice that the field $K$ is totally real.
We can find a nonzero vector $(z_1,\ldots,z_N) \in K^N \subset \RR^N$, such that
$\sum_{j=1}^N c_{1,j}z_j=0.$
Thus, for all $i$, we have
\[
\sum_{j=1}^N b_{i,1}c_{1,j}z_j = 0.
\]

We can find integers $y_1,\ldots,y_N$, not all zero, and a nonzero integer $y$
such that
\[
|y z_i-y_i| < \varepsilon' \text{ for $i=1,\cdots, N$.}
\]
For all $i=1,\cdots,N$, we have the upper bound
\begin{eqnarray*}
|\sum_{j=1}^N b_{i,1}c_{1,j}y_j|
&=& |y\cdot \sum_{j=1}^N b_{i,1}c_{1,j}z_j - \sum_{j=1}^N b_{i,1}c_{1,j}y_j| \\
&=& |\sum_{j=1}^N b_{i,1}c_{1,j}(y z_j-y_j)| \\
&\leq& \sum_{j=1}^N |b_{i,1}c_{1,j}|\cdot |y z_j-y_j| \\
&\leq& (\sum_{j=1}^N |b_{i,1}c_{1,j}|) \cdot \varepsilon' \\
&\leq& N R\cdot \varepsilon'.
\end{eqnarray*}

Let $P=(2^{y_1},\ldots,2^{y_N})$, then we have 
\begin{eqnarray*}
\hat{h}_{A}^{+}(P)
&=& \limsup_{n\rightarrow \infty} \frac{1}{\lambda_1^n} h(\varphi_A^n(P)) \\
&=& \limsup_{n\rightarrow \infty} \sum_{v \in M_K} \max_{1 \leq i \leq N}\{ 0, \sum_{j=1}^N
    \frac{\sum_{k=1}^N b_{i,k}c_{k,j}\lambda_k^n}{\lambda_1^n} \log\|2^{y_j}\|_v \} \\
&=& \sum_{v \in M_K} \max_{1 \leq i \leq N} \{ 0, \sum_{j=1}^N b_{i,1}c_{1,j} \log\|2^{y_j}\|_v \} \\
&\leq& 2\log(2) \max_{1 \leq i \leq N} \{ |\sum_{j=1}^N b_{i,1}c_{1,j} y_j)| \} \\
&\leq& 2\log(2)N R\cdot \varepsilon'.
\end{eqnarray*}

Let $\varepsilon = \frac{\varepsilon'}{2\log(2)NR}$, then $\hat{h}_{A}^{+}(P)<\varepsilon$.
Also, observe that $P$ is not preperiodic. Hence, by
\cite[Corollary 31]{S1}, $\hat{h}_{A}^{+}(P)>0$. This completes the proof.
\end{proof}

\begin{remark}
The proposition also implies that a lower bound of Lehmer type does not exist for $\hat{h}_{A}^{+}$.
\end{remark}


\section{The Total Canonical Height Function}
\label{sec:tot:can:ht}

In this section, we construct a modified version of canonical height functions for monomial maps.
Our method is motivated by the construction in \cite{S2,K1,K2,Lee}. The function we construct is called the
{\em total canonical height} function, and it has all the properties we want for two-dimensional semisimple monomial maps
and certain cases of three dimensional monomial maps.

First, we are going to define the backward canonical height.
For $A\in \Mat^+_N(\mathbb{Z})$, let
\[
A'=|\det(A)|\cdot A^{-1} = \text{sgn}(\det (A))\cdot \text{ad}(A),
\]
where $\text{sgn}(.)$ is the sign function,
and $\text{ad}(A)$ is the classical adjoint matrix of $A$.
Notice that $A' \in \Mat^+_N(\mathbb{Z})$, and if $\det(A)\ne 0$, then
$\det(A')=\det(A)^{N-1}$ is also nonzero.

\begin{definition}
The {\em backward canonical height of $P\in\GG_m^N(\Qbar)$ with respect to $\varphi_A$} is defined as
\[
\hat{h}_{A}^{-}(P):=\hat{h}_{\varphi_{A'}}^{+}(P).
\]
\end{definition}

The motivation of the definition is as follows. First, we would want to define $\hat{h}_{A}^{-}(P)$ as
$\hat{h}_{A^{-1}}^{+}(P)$, but in general, $A^{-1}\in\Mat_N(\QQ)$ may not have integer entries. This
means, in particular, that $\varphi_{A^{-1}}$ involves taking roots of complex numbers, and is a multi-valued
function. However, for any two $Q, Q'$ with $\varphi_A(Q)=\varphi_A(Q')=P$, there are roots of unities
$\zeta_1,\cdots,\zeta_N$ such that $Q\cdot (\zeta_1,\cdots,\zeta_N)=Q'$. Thus $h(Q)=h(Q')$, and the
height $h(\varphi_{A^{-1}}(P))$ is indeed well-defined.

If we try to define $\hat{h}_A^-$ using $h(\varphi_{A^{-1}}(P))$, we will still encounter the problem
of what the dynamical degree of $\varphi_A^{-1}$ should be. One might be able to settle this problem
by applying the language of correspondences and the theory dynamics of correspondences, see,
 for example, the work of Dinh and Sibony~\cite{DS}.
But for monomial maps, one can avoid this problem by the following observation.

Notice that,
in the Example~\ref{ex:easy}, for a positive integer $d$, taking $d$-th power
to each coordinate does not change
the canonical height. Also notice that by Example~\ref{ex:easy_too},
taking a negative power will change the canonical height.
Thus, we let $d=|\det(A)|\ge 1$ and obtain the matrix
$A'$ with integer entries.

In fact, the matrix $A'$ is also used by the first author to show a duality for
pullback map on complimentary dimension for monomial maps, see \cite[Proposition 3.1]{Lin2}.
We will show in the following that this definition indeed gives
the desired property for canonical height functions as well.

Finally, we define the {\em total canonical height} by
\[
\hat{h}_{A}(P):=\hat{h}_{A}^{+}(P)+\hat{h}_{A}^{-}(P).
\]

The following standard properties of $\hat{h}_{A}$ can be deduced from the corresponding properties for $\hat{h}_{A}^{+}$.
\begin{proposition}
For $A \in \Mat_N^+(\mathbb{Z})$, we have
\begin{enumerate}
\item $0\le \hat{h}_A(P) <\infty$.
\item  Suppose the characteristic polynomial of $A$ is irreducible and $\rho(A)>1$ .
Then for $P \in \GG_m^N(\Qbar)$,
\[
\hat{h}_{A}(P)=0 \Longleftrightarrow P \in  \preper(\varphi_A)
\]
\end{enumerate}
\end{proposition}

\begin{proof}
Part (1) is a direct consequence of the corresponding property for $\hat{h}_A^+$ and $\hat{h}_{A'}^+$.

For part (2), if $\hat{h}_{A}(P)= \hat{h}_{A}^+(P) + \hat{h}_{A}^-(P)=0$, then $\hat{h}_{A}^+(P)=0$. By \cite[Corollary 31]{S1},
we know $P \in  \text{PrePer}(\varphi_A)$.
Conversely, the given conditions on $A$ implies that both eigenvalues of $A$ are not roots of unity.
So by \cite[Proposition 20(d)]{S1}, all coordinate of $P$ are roots of unity.
This means $P$ is a preperiodic point for both $\varphi_{A}$ and $\varphi_{A'}$, so
\[
\hat{h}_{A}^+(P)=\hat{h}_{A}^-(P)=0.
\]
Hence $\hat{h}_{A}(P)=0.$
\end{proof}

\begin{theorem}
\label{thm:unique}
Suppose $A \in \Mat_N^+(\mathbb{Z})$ satisfies $\rho(A)>1$, and denote $D=|\det(A)|=|\lambda_1\cdots\lambda_N|$.
Then
\begin{enumerate}
  \item For all $P \in \GG_m^N(\Qbar)$, we have
\[
\frac{\hat{h}_{A}(\varphi_{A}(P))}{|\lambda_1\lambda_N|} + \frac{\hat{h}_{A} ( \varphi_{A'}(P) )}{D}
             =\Bigl(\frac{1}{|\lambda_1|}+\frac{1}{|\lambda_N|}\Bigr)\cdot \hat{h}_{A}(P).
\]
  \item Moreover, $\hat{h}_{A}$ enjoys the following uniqueness property:
if $\hat{h'}$ is another function satisfying $(1)$ and $\hat{h'}=\hat{h}_{A}+O(1)$,
then $\hat{h'}=\hat{h}_{A}$.
\end{enumerate}
\end{theorem}

\begin{remark}
Notice that, when $N=2$, the equality in (1) takes the simple form:
\[
\hat{h}_{A}(\varphi_{A}(P)) + \hat{h}_{A} ( \varphi_{A'}(P) )
             =\bigl(|\lambda_1|+|\lambda_2|\bigr)\cdot \hat{h}_{A}(P).
\]
\end{remark}

\begin{proof}
First, for all $P \in \GG_m^N(\Qbar)$, the following holds:
\[
\varphi_A(\varphi_{A'}(P))=\varphi_{D\cdot I_N}(P)=\varphi_{A'}(\varphi_A(P)),
\]
also notice that
\[
\delta_A=|\lambda_1|,\ \delta_{A'}=|\lambda_1\ldots\lambda_{N-1}|.
\]
As a consequence, one has
\begin{eqnarray*}
& & \hat{h}_{A}^{+} ( \varphi_{A'}(P))\\
&=& \limsup_{n\to\infty} \frac{h(\varphi_A^n(\varphi_{A'}(P)))}{n^\ell\cdot\delta_A^n}
 =  \limsup_{n\to\infty} \frac{h(\varphi_{A^n\cdot A'}(P))}{n^\ell\cdot\delta_A^n} \\
&=& \limsup_{n\to\infty} \frac{h(\varphi_{D\cdot A^{n-1}}(P))}{n^\ell\cdot\delta_A^n}
 =  \limsup_{n\to\infty} \frac{h(\varphi_A^{n-1}(P)^D)}{n^\ell\cdot\delta_A^n} \\
&=& \limsup_{n\to\infty} \frac{D}{\delta_A} \cdot \frac{(n-1)^\ell}{n^\ell}\cdot\frac{h(\varphi_A^{n-1})(P)}{(n-1)^\ell\cdot \delta_A^{n-1}}
 =  \frac{D\cdot \hat{h}_{A}^{+} (P) }{\delta_A} \\
&=& \frac{D\cdot \hat{h}_{A}^{+} (P) }{|\lambda_1|}.
\end{eqnarray*}
Similarly,
\[
\hat{h}_{A'}^{+} ( \varphi_{A}(P)) = \frac{D\cdot \hat{h}_{A}^{+} (P) }{\delta_{A'}}
  = |\lambda_N|\cdot\hat{h}_{A}^{+} (P) .
\]
Therefore,
\begin{eqnarray*}
& & \frac{\hat{h}_{A}(\varphi_{A}(P))}{|\lambda_1\lambda_N|} + \frac{\hat{h}_{A} ( \varphi_{A'}(P) )}{D} \\
&=& \frac{\hat{h}_{A}^{+} ( \varphi_{A}(P))}{|\lambda_1\lambda_N|} +
    \frac{\hat{h}_{A'}^{+} ( \varphi_A(P))}{|\lambda_1\lambda_N|}+
    \frac{\hat{h}_{A}^{+} ( \varphi_{A'}(P))}{D} +
    \frac{\hat{h}_{A'}^{+} ( \varphi_{A'}(P))}{D}\\
&=& \frac{\hat{h}_{A}^{+}(P)}{|\lambda_N|} +
    \frac{\hat{h}_{A'}^{+} (P)}{|\lambda_1|}+
    \frac{\hat{h}_{A}^{+} (P)}{|\lambda_1|} +
    \frac{\hat{h}_{A'}^{+}(P)}{|\lambda_N|}\\
&=& \Bigl(\frac{1}{|\lambda_1|}+\frac{1}{|\lambda_N|}\Bigr)\cdot \hat{h}_{A}(P).
\end{eqnarray*}

We claim that
\[
\frac{1}{|\lambda_1\lambda_N|}+\frac{1}{D}
< \frac{1}{|\lambda_1|}+\frac{1}{|\lambda_N|}.
\]
Assuming the claim, let us prove the uniqueness of $\hat{h}$.
Suppose $\hat{h'}$ is another function with properties (1) such that
$g:=\hat{h}_{A}-\hat{h'}$ is bounded on $\GG_m^2(\Qbar)$.
Let
$$
M:=\sup_{P \in \GG_m^N(\Qbar)} |g(P)|.
$$
Then
\begin{eqnarray*}
& & \Bigl(\frac{1}{|\lambda_1\lambda_N|}+\frac{1}{D}\Bigr)M
<\Bigl(\frac{1}{|\lambda_1|}+\frac{1}{|\lambda_N|}\Bigr) M  \\
& = &\sup_{P \in \GG_m^N(\Qbar)}\Bigl| \Bigl(\frac{1}{|\lambda_1|}+\frac{1}{|\lambda_N|}\Bigr)\cdot g(P)\Bigr|\\
& = &
\sup_{P \in \GG_m^N(\Qbar)} \Bigl| \frac{1}{|\lambda_1\lambda_N|}g ( \varphi_A(P) ) + \frac{1}{D}g(\varphi_{A'}(P)) \Bigr|\\
&\leq& \Bigl(\frac{1}{|\lambda_1\lambda_N|}+\frac{1}{D}\Bigr)M.
\end{eqnarray*}
This implies $M=0$, hence $\hat{h'}=\hat{h}_{A}.$

It remains to prove the claim. Notice that it is equivalent to the following inequality.
\[
D+|\lambda_1\lambda_N|
< D\cdot\bigl(|\lambda_1|+|\lambda_N|\bigr).
\]
There are three cases to be considered:
\[
D \ge |\lambda_1|,\  |\lambda_1|>D>|\lambda_N|,\text{ and }|\lambda_N|\ge D.
\]

For the first case, we have $D|\lambda_1|>D$ (since $|\lambda_1|>1)$ and $D|\lambda_N| \ge |\lambda_1\lambda_N|$.
Thus $D\cdot\bigl(|\lambda_1|+|\lambda_N|\bigr) > D+|\lambda_1\lambda_N|$.

As of the second case, we have $(|\lambda_1|-D)(|\lambda_N|-D)<0$ and
since $D\in\NN$, we know that
$-D^2+D \leq 0$.
Thus we have
\[
(|\lambda_1|-D)(|\lambda_N|-D)-D^2+D < 0,
\]
which implies the required inequality.

In the last case,
we have the inequality $D=\prod_{i=1}^N |\lambda_i| \geq D^N$. Thus $D = 1$.
Therefore $(|\lambda_1|-1)(|\lambda_N|-1)<0$ because $\rho(A)>1$.
Hence $ 1 + |\lambda_1\lambda_N| < |\lambda_1|+|\lambda_N|$ and this completes the proof of the claim.
\end{proof}

For a point $P\in\GG_m^N(\Qbar)$, one might want to know how fast the height function
grows in the orbit $\{\varphi^n(P)\}_{n=1}^\infty$.
As an application of the above theorem, we obtain a linear recurrence relation of the height sequence
$\{\hat{h}_A(\varphi_A^n(P))\}_{n=1}^\infty$ in the following.

\begin{corollary}
Suppose $A \in \Mat_N^+(\mathbb{Z})$ satisfies $\rho(A)>1$ and $P \in \GG_m^N(\Qbar)$.
Then
\[
\hat{h}_A(\varphi_A^{n+2}(P))
-(|\lambda_1|+|\lambda_N|)\cdot \hat{h}_A(\varphi_A^{n+1}(P))
+|\lambda_1\lambda_N|\cdot \hat{h}_A(\varphi_A^{n}(P))
=0
\]
for all $n\ge 0$.
Moreover, we have
\[
\hat{h}_A(\varphi_A^{n}(P))=
\frac{|\lambda_1\lambda_N^n|-|\lambda_1^n\lambda_N|}{|\lambda_1|-|\lambda_N|}\cdot \hat{h}_A(P)
+\frac{|\lambda_1^n|-|\lambda_N^n|}{|\lambda_1|-|\lambda_N|}\cdot \hat{h}_A(\varphi_A(P)).
\]
\end{corollary}

\begin{proof}
Notice that
\[
\hat{h}_A(\varphi_{D\cdot I_N}(P)) = D\cdot \hat{h}_A(P).
\]
If we replace the $P$ in Theorem~\ref{thm:unique} by $\varphi_A^{n+1}(P)$.
Then, we have the following
\begin{eqnarray*}
0
& = & \frac{\hat{h}_{A}(\varphi_A^{n+2}(P))}{|\lambda_1\lambda_N|}
    -\Bigl(\frac{1}{|\lambda_1|}+\frac{1}{|\lambda_N|}\Bigr)\cdot \hat{h}_{A}(\varphi_A^{n+1}(P))\\
&   &  + \frac{\hat{h}_{A} ( \varphi_{A'}(\varphi_A^{n+1}(P)) )}{D}\\
& = & \frac{\hat{h}_{A}(\varphi_A^{n+2}(P))}{|\lambda_1\lambda_N|}
    -\Bigl(\frac{1}{|\lambda_1|}+\frac{1}{|\lambda_N|}\Bigr)\cdot \hat{h}_{A}(\varphi_A^{n+1}(P))+
    \hat{h}_{A}(\varphi_A^{n}(P)).
\end{eqnarray*}
Hence
\[
\hat{h}_A(\varphi_A^{n+2}(P))
-(|\lambda_1|+|\lambda_N|)\hat{h}_A(\varphi_A^{n+1}(P))
+|\lambda_1\lambda_N| \hat{h}_A(\varphi_A^{n}(P))
=0.
\]

Define the sequence $H_n:=\hat{h}_A(\varphi_A^{n}(P))$, then $H_n=c_1|\lambda_1|^n+c_2|\lambda_N|^n$ by the recurrence relation.
We can solve $c_1$ and $c_2$ by the initial data $H_0=\hat{h}_A(P),H_1=\hat{h}_A(\varphi_A(P))$, and obtain that
\[
c_1=\frac{H_1-|\lambda_N|H_0}{|\lambda_1|-|\lambda_N|},\quad
c_2=\frac{|\lambda_1|H_0-H_1}{|\lambda_1|-|\lambda_N|}.
\]
This gives the above expression for $H_n=\hat{h}_A(\varphi_A^{n}(P))$.
\end{proof}

\begin{theorem}
\label{theorem:LB1}
Suppose $A \in \Mat_N^+(\ZZ)$ is diagonalizable.
\begin{enumerate}
\item If all the eigenvalues of $A$ have the same modulus, then
\[
\hat{h}_A^+(P)\geq h(P).
\]
\item If the eigenvalues of $A$ have only two different modulus, then
\[
\hat{h}_A(P)\geq h(P).
\]
\end{enumerate}
\end{theorem}

\begin{proof}
Write the matrix $A$ as $A=B\Lambda B^{-1}$, where
\[
\Lambda=\text{diag}(\lambda_1,\ldots,\lambda_N)
\]
is diagonal with $|\lambda_1|\geq\ldots\geq|\lambda_N|$.

For part (1), we have $|\lambda|=|\lambda_1|=\ldots=|\lambda_N|$.
We can find a sequence of positive integer $\{n_r\}_{r=1}^\infty$ such that
\[
A^{n_r}/|\lambda|^{n_r} \rightarrow I_N.
\]
Therefore, we have the following:
\begin{eqnarray*}
\hat{h}_A^+(P)
& = & \limsup_{n \rightarrow \infty}\frac{h(\varphi_A^n(P))}{|\lambda|^n}\\
& \geq & \lim_{r \rightarrow \infty}\frac{h(\varphi_A^{n_r}(P))}{|\lambda|^{n_r}}\\
& = & \lim_{r \rightarrow \infty} \sum_{v \in M_K} \max_{1\leq i \leq N}
        \Bigl\{0, \sum_{j=1}^N \frac{a_{i,j}(n_r)}{|\lambda|^{n_r}}\log\|x_j\|_v\Bigr\} \\
& = & \sum_{v \in M_K} \max_{1\leq i \leq N} \{0, \log\|x_i\|_v\} \\
& = & h(P).
\end{eqnarray*}

For (2), assume $|\lambda_1|=\ldots=|\lambda_{i_0}|=|\lambda|>|\mu|=|\lambda_{i_0+1}|=\ldots=|\lambda_N|$
where $1\le i_0<N$.
We can find a sequence of positive integer $\{n_r\}_{r=1}^\infty$ such that
\[
A^{n_r}/|\lambda|^{n_r} \rightarrow B \;
\diag(\underbrace{1,\ldots,1}_{i_0},\underbrace{0,\ldots,0}_{N-i_0})
B^{-1}=(\alpha_{i,j}).
\]
Therefore, we have the following:
\begin{eqnarray*}
\hat{h}_A^+(P)
& = & \limsup_{n \rightarrow \infty}\frac{h(\varphi_A^n(P))}{|\lambda|^n}\\
& \geq & \lim_{r \rightarrow \infty}\frac{h(\varphi_A^{n_r}(P))}{|\lambda|^{n_r}}\\
& = & \lim_{r \rightarrow \infty} \sum_{v \in M_K} \max_{1\leq i \leq N}
        \{0, \sum_{j=1}^N \frac{a_{i,j}(n_r)}{|\lambda|^{n_r}}\log\|x_j\|_v\} \\
& = & \sum_{v \in M_K} \max_{1\leq i \leq N} \{0, \sum_{j=1}^N \alpha_{i,j}\log\|x_j\|_v\}.
\end{eqnarray*}
Notice that $\delta_{A'}=|\lambda^{i_0}\mu^{N-i_0-1}|$.
We can find another sequence of positive integers $\{n_s\}_{s=1}^\infty$ such that
\[
A'^{n_s}/\delta_{A'}^{n_s} \rightarrow B
\diag(\underbrace{0,\ldots,0}_{i_0},\underbrace{1,\ldots,1}_{N-i_0})
B^{-1}=(\beta_{i,j}).
\]
Therefore, we also have
\begin{eqnarray*}
\hat{h}_A^-(P)
& = & \hat{h}_{A'}^+(P) \\
& = & \limsup_{n \rightarrow \infty}\frac{h(\varphi_{A'}^n(P))}{\delta_{A'}^n}\\
& \geq & \lim_{r \rightarrow \infty}\frac{h(\varphi_{A'}^{n_r}(P))}{\delta_{A'}^{n_r}}\\
& = & \lim_{r \rightarrow \infty} \sum_{v \in M_K} \max_{1\leq i \leq N}
        \{0, \sum_{j=1}^N \frac{a'_{i,j}(n_r)}{\delta_{A'}^{n_r}}\log\|x_j\|_v\} \\
& = & \sum_{v \in M_K} \max_{1\leq i \leq N} \{0, \sum_{j=1}^N \beta_{i,j}\log\|x_j\|_v\}.
\end{eqnarray*}
One has the relation $(\alpha_{i,j})+(\beta_{i,j})=I_N$, hence
\begin{eqnarray*}
\hat{h}_A(P)
& = & \hat{h}_A^+(P)+\hat{h}_A^-(P) \\
& \geq & \sum_{v \in M_K} \max_{1\leq i \leq N} \{0, \sum_{j=1}^N \alpha_{i,j}\log\|x_j\|_v\}+\\
& &        \sum_{v \in M_K} \max_{1\leq i \leq N} \{0, \sum_{j=1}^N \beta_{i,j}\log\|x_j\|_v\}\\
& \geq & \sum_{v \in M_K} \max_{1\leq i \leq N} \{0, \sum_{j=1}^N (\alpha_{i,j}+\beta_{i,j})\log\|x_j\|_v\}\\
& \geq & \sum_{v \in M_K} \max_{1\leq i \leq N} \{0, \log\|x_i\|_v\}\\
& = & h(P)
\end{eqnarray*}
This concludes the proof.
\end{proof}

The bound in Theorem~\ref{theorem:LB1} has some nice consequences.

\begin{corollary}[Northcott finiteness property]
If $A$ satisfies either (1) or (2) of Theorem~\ref{theorem:LB1}, then
for a bounded degree and bounded total canonical height, there are only finitely many points in $\GG_m^N(\Qbar)$
within these bounds. More precisely, given any $B>0$ and $D>0$, we have
\[
\#\Bigl\{P\in\GG_m^N(\Qbar) \Bigl| [\QQ(P):\QQ]\le D\text{ and } \hat{h}_A(P)\le B \Bigr\} < \infty.
\]
\end{corollary}

Since the total canonical height function is bounded below by the usual height function, any lower bound on
$h(P)$ will immediately induce a lower bound on $\hat{h}_A(P)$. This is summarized in the following corollary.
For more about Lehmer type lower bounds, see~\cite{AmDv, do}.
\begin{corollary}
If $A$ satisfies either (1) or (2) of Theorem~\ref{theorem:LB1}, then any Lehmer type lower bound for
$h(P)$ will induce a Lehmer type bound for $\hat{h}_A(P)$,
and the bound does not depend on $A$.
\end{corollary}

Before we state the next property for $\hat{h}_A$, we need to introduce some notation.
Two real-valued functions $\lambda$ and $\lambda'$ on $\GG_m^N(\Qbar)$ are said to be equivalent
if there exist positive constant $C_1, C_2$ such that
$$
C_1\lambda(x) \leq \lambda'(x) \leq C_2\lambda(x) \text{ for all $x \in \GG_m^N(\Qbar)$}.
$$
We use the notation $\lambda \asymp \lambda'$ to denote this equivalence.

\begin{corollary}
\label{prop:asymp}
Suppose $A \in \Mat_N^+(\mathbb{Z})$ satisfies either (1) or (2) of Theorem~\ref{theorem:LB1}, and $\rho(A)>1$.
Then $\hat{h}_{A} \asymp h$ on $\GG_m^N(\Qbar)$.
\end{corollary}

\begin{proof}
By Theorem \ref{theorem:LB1},
$h(P) \leq \hat{h}_{A}(P) \text{ for all } P \in \GG_m^N(\Qbar)$, so we can simply let $C_1(A)=1$.
Also, there exists a $C_2(A)$ by \cite[Proposition 24]{S1} such that
$$
\hat{h}_{A}^\pm(P) \leq C_2(A)h(P) \text{ for all } P \in \GG_m^N(\Qbar).
$$
Hence, we proved that $\hat{h}_{A} \asymp h$ on $\GG_m^N(\Qbar)$.
\end{proof}

Conditions (1) or (2) in Theorem~\ref{theorem:LB1} are quite restricted in general. However, they include
(almost) all diagonalizable
cases in dimension two, and some major cases in dimensions three and four.

\begin{corollary}
Let $A \in \Mat_2^+(\mathbb{Z})$ be diagonalizable, and $\rho(A)>1$,
then $\varphi_A$ satisfies all the conclusions from Theorem~\ref{theorem:LB1} to Corollary~\ref{prop:asymp}.
\end{corollary}

\begin{corollary}
\mbox{}
\begin{enumerate}
\item Given $A\in\Mat_3^+(\ZZ)$, suppose that $A$ has complex eigenvalues and $\rho(A)>1$,
then $\varphi_A$ satisfies all the conclusions from Theorem~\ref{theorem:LB1} to Corollary~\ref{prop:asymp}.
\item Suppose $A\in\Mat_4^+(\ZZ)$ is diagonalizable, with eigenvalues two pairs of conjugate complex numbers, and $\rho(A)>1$,
then $\varphi_A$ satisfies all the conclusions from Theorem~\ref{theorem:LB1} to Corollary~\ref{prop:asymp}.
\end{enumerate}
\end{corollary}


\section{Monomial Maps associated to Non-diagonalizable Matrices}
\label{sec:non_diag}

\subsection{Tow-dimensional non-diagonalizable matrices}

Both the canonical height function and the totally canonical height function will encounter some problems
in the non-diagonalizable cases, even in dimension two. We will first illustrate the problems by a concrete example, then we
will show what happens in general.

\begin{example}
Consider
$A=
\left(
  \begin{array}{cc}
    2 & 1 \\
    0 & 2 \\
  \end{array}
\right)
, P=(x,y)$, so $\varphi_A(x,y)=(x^2y, y^2)$.
We have the following
\[
\hat{h}_A^+(P) = \frac{h(y)}{2} = \hat{h}_A^-(P) \text{ and }
\hat{h}_A(P) = 2 \hat{h}_A^+(P) = h(y).
\]
We observe the following two phenomena.
\begin{enumerate}
  \item  If $P=(x,y)$ with $y$ a root of unity, for instance, $P=(2,1)$,
         then $\hat{h}_A(P)=\hat{h}_A^+(P)=0$ but $P$ may not be a preperiodic point.
  \item Since both the canonical height and the total canonical height only depend on $y$, the Northcott finiteness property
        obviously fail in this example.
\end{enumerate}
\end{example}

We remark that for (1), given a non-diagonalizable $2\times 2$ matrix $A$ with $\rho(a)>1$, the points $P\in\GG_m^2(\Qbar)$ such that
$ \hat{h}_{A}^{+}(P)=0 $ is characterized by Silverman.

\begin{theorem}[Silverman~\cite{S1}]
\label{thm_zero_height_by_Silverman_in_dim_2}
If $A=
\left(
  \begin{array}{cc}
    a & b \\
    c & d \\
  \end{array}
\right)
\in\Mat_2^+(\ZZ)$ is not diagonalizable, and $\lambda$ is the eigenvalue of $A$, then
$\hat{h}_{A}^{+}(P)=0$
if and only if one of the following conditions is true:
\begin{enumerate}
  \item[(i)] $P \in \preper(\varphi_A)$
  \item[(ii)] The coordinates of $P=(x,y)$ satisfy:
  \[
  \left\{
    \begin{array}{ll}
      x^{c}y^{d-\lambda} \text{ is a root of unity,} & \hbox{if $c \neq 0$;} \\
      x^{a-\lambda}y^{b} \text{ is a root of unity,} & \hbox{if $c = 0$.}
    \end{array}
  \right.
  \]
\end{enumerate}
\end{theorem}

For a proof, see \cite[Theorem~36]{S1}.

Following the general principle that ``Geometry Determines Arithmetic'', as stated in \cite[p.2]{HiSi}, we will study
the geometry of the map to see why this happens.

Notice that if we define the projection map $\pi_y : (\CC^*)^2 \to \CC^*$, $(x,y)\mapsto y$, then the monomial map $f_A$
is {\em semi-conjugate} to the power map $\varphi_2:y\mapsto y^2$, as illustrated in the following diagram.
\[
\xymatrix{
(\CC^*)^2\ar[d]_{\pi_y}\ar[r]^{\varphi_A}& (\CC^*)^2\ar[d]^{\pi_y}\\
\CC^*\ar[r]_{\varphi_2}  & \CC^*
}
\]
This means, the map $\varphi_A$ preserves the fibration defined by the map $\pi_y$. Then we have
$\hat{h}_A(P) = 2 \hat{h}_A^+(P) = h(\pi_y(P))$. That is, the height function actually only
captures the height growth behavior on the base of the fibration.
As a consequence, the points $P\in\GG_m^2(\Qbar)$
with $\hat{h}(P)=0$ are exactly those $P$ such that $\pi_2(P)$ is a preperiodic point.
Also, the set of points whose canonical height is bounded by $D>0$ contains all fibers $\pi_y^{-1}(y_0)$
such that $h(y_0)\le D$.

\subsection{Higher dimension}
In order to generalize the above observation, we assume $A\in\Mat_N^+(\ZZ)$ has only one eigenvalue $\lambda\in\ZZ$.
If $A$ is diagonalizable, then $A=\lambda\cdot I_N$. By Example~\ref{ex:easy} and \ref{ex:easy_too}, one obtains that
\begin{eqnarray*}
\hat{h}_{A}^{+}(P) = \hat{h}_{A}^{-}(P) &=&
\left\{
  \begin{array}{ll}
    h(P), &  \text{if $\lambda>0$;} \\
    \max\{h(P),h(P^{-1})\}, & \text{if $\lambda<0$.}
  \end{array}
\right.
\end{eqnarray*}

Now, suppose that $A$ is not diagonalizable, so by \cite[Theorem 6.2]{Lin1}, $(\ell_A+1)$ is the size of the largest
Jordan block of $A$. To simplify the notation, we write $\ell=\ell_A$.
Assume that there are $m$ Jordan blocks of size $(\ell+1)$, then the matrix $(A-\lambda I)^{\ell}$
has rank $m$. Let
\[
\pi = \varphi_{(A-\lambda I)^{\ell}} : (\CC^*)^N \to (\CC^*)^N.
\]
The monomial map $\pi$ is not dominant, and the image is an $m$-dimensional subtorus of $(\CC^*)^N$, denoted by $T$.
Moreover, since $A$ commutes with $A-\lambda I$, we know that the map $\varphi_A|_T$ of $\varphi_A$ restricting on $T$ is
surjective onto $T$, and the following diagram commutes.
\[
\xymatrix{
(\CC^*)^N\ar[d]_{\pi}\ar[r]^{\varphi_A}& (\CC^*)^N\ar[d]^{\pi}\\
 T \ar[r]_{\varphi_A|_T}  & T
}
\]
Geometrically, this means that the map $\varphi_A$ preserves the fibration defined by $\pi$. The next theorem shows that
the arithmetic of canonical height is indeed controlled by this fibration.

\begin{theorem}
\label{thm:non-diag-ht}
Under the above assumption and notation, we have
\begin{enumerate}
\item If $\lambda > 0$, then $\hat{h}_{A}^{+}(P) = \frac{h(\pi(P))}{\ell !\lambda^{\ell}}$, and
       $\hat{h}_{A}^{-}(P) = \frac{h(\pi(P)^{-1})}{\ell !\lambda^{\ell}}$.
\item If $\lambda < 0$, then $\hat{h}_{A}^{+}(P) = \hat{h}_{A}^{-}(P) = \frac{\max\{h(\pi(P)),h(\pi(P)^{-1})\}}{\ell !|\lambda|^{\ell}}$.
\end{enumerate}
\end{theorem}

\begin{proof}
First, we will prove the theorem for $\hat{h}_{A}^{+}$ in the case $\lambda > 0$.
Under the assumption, we can write $A=\lambda I + \Ncal$, where $\Ncal$ is a nilpotent matrix
such that $\Ncal^\ell\ne 0$ but $\Ncal^{\ell+1}=0$. Then
\[
A^n = (\lambda I + \Ncal)^n
    = \sum_{k=0}^\ell \binom{n}{k} \lambda^{n-k} \Ncal^k = \sum_{k=0}^\ell
                           \frac{n^k+\text{(lower order terms)}}{k!}\cdot \lambda^{n-k} \Ncal^k.
\]
Thus, as $n\to\infty$, we have
\begin{eqnarray*}
\frac{A^n}{n^\ell\lambda^n}
    &=& \sum_{k=0}^\ell \binom{n}{k} \lambda^{n-k} \Ncal^k = \sum_{k=0}^\ell
                           \frac{n^k+\text{(lower order terms)}}{n^\ell\cdot k!\cdot\lambda^k}\cdot  \Ncal^k\\
    &\longrightarrow &  \frac{\Ncal^\ell}{\ell ! \lambda^\ell}.
\end{eqnarray*}
Also notice that $\Ncal = A-\lambda I$. Let $\Ncal^\ell = (b_{ij})$, then
\begin{eqnarray*}
\hat{h}_A^+(P)
& = & \limsup_{n \rightarrow \infty}\frac{h(\varphi_A^n(P))}{n^\ell \lambda^n}\\
& = & \limsup_{n \rightarrow \infty} \sum_{v \in M_K} \max_{1\leq i \leq N}
        \{0, \sum_{j=1}^N \frac{a_{i,j}(n)}{n^\ell\lambda^{n}}\log\|x_j\|_v\} \\
& = & \sum_{v \in M_K} \max_{1\leq i \leq N} \{0, \sum_{j=1}^N \frac{b_{i,j}}{\ell !\lambda^{\ell}}\log\|x_j\|_v\}\\
& = & \frac{1}{\ell !\lambda^{\ell}} \sum_{v \in M_K} \max_{1\leq i \leq N} \{0, \sum_{j=1}^N \log\|\pi(P)_j\|_v\}\\
& = & \frac{h(\pi(P))}{\ell !\lambda^{\ell}}.
\end{eqnarray*}

If $\lambda < 0$, then the sequence $\{\frac{A^n}{n^\ell\lambda^n} \}_{n=1}^\infty$ has two limit points, namely,
$\frac{\Ncal^\ell}{\ell ! \lambda^\ell}$ and $-\frac{\Ncal^\ell}{\ell ! \lambda^\ell}$. As a consequence,
the sequence $\frac{h(\varphi_A^n(P))}{n^\ell \lambda^n}$ also has two limit points, that is,
$\frac{h(\pi(P))}{\ell !\lambda^{\ell}}$ and $\frac{h(\pi(P)^{-1})}{\ell !\lambda^{\ell}}$, and thus the limsup is the
maximum of the two.
The proof for $\hat{h}_A^-$ is similar. This completes the proof.
\end{proof}

A direct corollary of the theorem is a characterization of points of canonical height zero.

\begin{corollary}
Under the same assumption as the theorem, we have
\begin{eqnarray*}
  && \hat{h}_A(P)=0  \Longleftrightarrow  \hat{h}_A^+(P)=0  \Longleftrightarrow  \hat{h}_A^-(P)=0\\
  & \Longleftrightarrow & h(\pi(P))=0\\
& \Longleftrightarrow & \text{every coordinate of $\pi(P)$ is a root of unity.}
\end{eqnarray*}
\end{corollary}

\begin{proof}
If any of the $\hat{h}_A$, $\hat{h}_A^+$, or $\hat{h}_A^-$ is zero, then by the theorem, one of
$h(\pi(P))$, $h(\pi(P)^{-1})$ must be zero. By the Kronecker's theorem, every coordinate of $\pi(P)$ is a root of unity,
hence the other value is zero, too.
\end{proof}

Notice that for the case $N=2$, the corollary states exactly the same condition as in Theorem~\ref{thm_zero_height_by_Silverman_in_dim_2}.

To summarize, we observe that, by Theorem~\ref{thm:non-diag-ht}, the canonical height of $P$ under $\phi_A$ depend only on $\pi(P)$.
For a point $Q\in T(\Qbar)$, every point in the fiber $\pi^{-1}(Q)$ will have the same canonical height, so the
canonical height function degenerates to the height function on a subtorus (the base of the fibration).
As a consequence, the Northcott finiteness property does not hold in this case.


\section{Points with Small Total Canonical Height}
A problem for the total canonical height function is that there are still points with small height in
dimension $\ge 3$. Therefore, for higher dimension, a general satisfying theory of canonical height 
functions is still needed.

\begin{proposition}
Given $A\in\Mat_N^+(\ZZ), N\ge 3$, suppose its characteristic polynomial is irreducible, and all the eigenvalues are
distinct and positive. Then, for any $\varepsilon >0$, there are infinitely many $P\in\GG_m^N(\QQ)$ with
$0 < \hat{h}_{A}(P) < \varepsilon$.
\end{proposition}

The technique to prove this proposition is very similar to the proof of Proposition~\ref{prop:small_height}, thus we 
only give a sketch of the proof and omit some detail.

\begin{proof}[Sketch of the Proof]
The proof is very similar to the proof of Theorem~\ref{prop:small_height}, so we omit some details.

Write the matrix $A$ as $A=B\Lambda B^{-1}$, where
\[
\Lambda = \diag(\lambda_1,\cdots,\lambda_N)
\]
is diagonal with $\lambda_1> \cdots > \lambda_N>0$,
$B=(b_{i,j})$, and $B^{-1}=(c_{i,j})$.
Define
\[
R:=\max_{i,j,k}\{ |b_{i,k}c_{k,j}| \},
\]
which only depends on $A$.

Let $K$ be the splitting field of the characteristic polynomial of $A$, so $K$ is totally real.
We can find a nonzero vector $(z_1,\ldots,z_N) \in K^N \subset \RR^N$, such that
$\sum_{j=1}^N c_{1,j}z_j=0$ and $\sum_{j=1}^N c_{N,j}z_j=0.$
Thus, for all $i$, we have
$\sum_{j=1}^N b_{i,1}c_{1,j}z_j = 0$ and $\sum_{j=1}^N b_{i,N}c_{N,j}z_j = 0$.

We can find integers $y_1,\ldots,y_N$, not all zero, and a nonzero integer $y$
such that
\[
|y z_i-y_i| < \varepsilon' \text{ for $i=1,\cdots, N$.}
\]
For all $i=1,\cdots,N$, we have the upper bounds
\[
|\sum_{j=1}^N b_{i,1}c_{1,j}y_j| \le  NR\cdot \varepsilon' \quad\text{and}\quad
|\sum_{j=1}^N b_{i,N}c_{N,j}y_j| \le  NR\cdot \varepsilon'.
\]
Let $P=(2^{y_1},\ldots,2^{y_N})$, then
\begin{eqnarray*}
\hat{h}_{A}(P)
&=& \limsup_{n\rightarrow \infty} \frac{h(\varphi_A^n(P))}{\lambda_1^n} +
    \limsup_{n\rightarrow \infty} \frac{h(\varphi_{A'}^n(P))}{(\lambda_2 \ldots \lambda_N)^n} \\
&\leq& 2\log(2)\Bigl(\max_{1 \leq i \leq N} \{ |\sum_{j=1}^N b_{i,1}c_{1,j} y_j)| \}
               +\max_{1 \leq i \leq N} \{ |\sum_{j=1}^N b_{i,N}c_{N,j} y_j)| \}\Bigr) \\
&\leq& 4\log(2)NR\varepsilon'.
\end{eqnarray*}
Let $\varepsilon = \frac {\varepsilon'}{4\log(2)NR}$ and we are done.
\end{proof}



\begin{bibdiv}
\begin{biblist}


\bib{AmDv}{article}{
   author={Amoroso, Francesco},
   author={Dvornicich, Roberto},
   title={A lower bound for the height in abelian extensions},
   journal={J. Number Theory},
   volume={80},
   date={2000},
   number={2},
}


\bib{DS}{article}{
   author={Dinh, Tien-Cuong},
   author={Sibony, Nessim},
   title={Upper bound for the topological entropy of a meromorphic
   correspondence},
   journal={Israel J. Math.},
   volume={163},
   date={2008},
   pages={29--44},
   issn={0021-2172},
}



\bib{do}{article}{
   author={Dobrowolski, E.},
   title={On a question of Lehmer and the number of irreducible factors of a
   polynomial},
   journal={Acta Arith.},
   volume={34},
   date={1979},
   number={4},
}


\bib{HP}{article}{
   author={Hasselblatt, Boris},
   author={Propp, James},
   title={Degree-growth of monomial maps},
   journal={Ergodic Theory Dynam. Systems},
   volume={27},
   date={2007},
   number={5},
   pages={1375--1397},
   issn={0143-3857},
}

\bib{HiSi}{book}{
   author={Hindry, Marc},
   author={Silverman, Joseph H.},
   title={Diophantine geometry},
   series={Graduate Texts in Mathematics},
   volume={201},
   publisher={Springer-Verlag},
   place={New York},
   date={2000},
   pages={xiv+558},
   isbn={0-387-98975-7},
   isbn={0-387-98981-1},
}

\bib{K1}{article}{
   author={Kawaguchi, Shu},
   title={Canonical height functions for affine plane automorphisms},
   journal={Math. Ann.},
   volume={335},
   date={2006},
   number={2},
}

\bib{K2}{article}{
   title={Local and global canonical height functions for affine space regular automorphisms},
   author={Shu Kawaguchi},
   eprint={arXiv:0909.3573v1 [math.AG]}
}

\bib{Lee}{article}{
   title={An upper bound for the height for regular affine automorphisms of $\AA^n$},
   author={ChongGyu Lee},
   eprint={arXiv:0909.3107 [math.NT]}
}

\bib{Lin1}{article}{
   title={Algebraic stability and degree growth of monomial maps and polynomial maps},
   author={Jan-Li Lin},
   eprint={arXiv:1007.0253 [math.DS]}
}

\bib{Lin2}{article}{
   title={On Degree Growth and Stabilization of Three Dimensional Monomial Maps},
   author={Jan-Li Lin},
   eprint={arXiv:1204.6258 [math.DS]}
}

\bib{S1}{article}{
   title={Dynamical Degrees, Arithmetic Degrees, and Canonical Heights for
          Dominant Rational Self-Maps of Projective Space},
   author={Silverman, Joseph H.},
   eprint={arXiv:1111.5664 [math.NT]}
}


\bib{S2}{article}{
   author={Silverman, Joseph H.},
   title={Rational points on $K3$ surfaces: a new canonical height},
   journal={Invent. Math.},
   volume={105},
   date={1991},
   number={2},
   pages={347--373},
   issn={0020-9910},
}

\end{biblist}
\end{bibdiv}
\end{document}